\documentclass[10pt,a4paper]{amsart}
\usepackage{geometry}

\usepackage{etoolbox}
\apptocmd{\sloppy}{\hbadness 10000\relax}{}{}

\newcommand{\comment}[1]{} 

\usepackage{microtype}
\usepackage{babel}
\usepackage[utf8]{inputenc}
\usepackage{enumitem}
\usepackage{amsmath,amsthm,amssymb,amsfonts}
\usepackage{tikz-cd}
\usetikzlibrary{arrows}
\input xy
\xyoption{all}
\newdir{ >}{{}*!/-5pt/@{>}}

\newcommand{\C}{\mathbb{C}}

\newcommand{\Z}{\mathbb{Z}}
\newcommand{\N}{\mathbb{N}}

\newcommand{\BF}{\mathfrak{BF}}

\newcommand{\cR}{\mathcal R}

\newcommand{\topdf}{\texorpdfstring}
\DeclareMathOperator{\reg}{reg}
\DeclareMathOperator{\coker}{coker}

\def\gr{\operatorname{gr}}

\usepackage{xcolor}

\definecolor{thmcol}{RGB}{145, 1, 1}
\definecolor{citecol}{RGB}{145, 1, 1}
\definecolor{linkcol}{RGB}{145, 1, 1}
\definecolor{urlcol}{RGB}{145, 1, 1}

\numberwithin{equation}{section}

\theoremstyle{plain}
\newtheorem{thm}[equation]{Theorem}
\newtheorem{lemma}[equation]{Lemma}
\newtheorem{lem}[equation]{Lemma}
\newtheorem{coro}[equation]{Corollary}
\newtheorem{prop}[equation]{Proposition}

\theoremstyle{definition}
\newtheorem{defn}[equation]{Definition}

\newtheorem{stan}[equation]{Standing assumption}

\theoremstyle{remark}
\newtheorem{rmk}[equation]{Remark}
\newtheorem{nota}[equation]{Notation}
\newtheorem*{ack}{Acknowledgements}
\usepackage[colorlinks=true, pagebackref=true]{hyperref}
\hypersetup{
    colorlinks,
    citecolor=citecol,
    linkcolor=linkcol,
    urlcolor=urlcol
}
\usepackage{amsrefs}
\usepackage{url}

\title[Nonexistence of graded unital homomorphisms]{Nonexistence of graded unital homomorphisms between Leavitt algebras and their Cuntz splices}
\author{Guido Arnone and Guillermo Corti\~nas}
\address{Dep. Matem\'atica-IMAS\\ Facultad de Ciencias Exactas y Naturales\\
Universidad de Buenos Aires\\ Argentina}
\thanks{The second named author was supported by CONICET and partially supported by grants PICT 2017-1395 from Agencia Nacional de Promoci\'on Cient\'\i fica y T\'ecnica, UBACyT 0150BA from Universidad de Buenos Aires, and PGC2018-096446-B-C21 from the Spanish Ministerio de Ciencia e Innovaci\'on}

\date{}

\begin{document}

\begin{abstract}
Let $n\ge 2$, let $\cR_n$ be the graph consisting of one vertex and $n$ loops and let $\cR_{n^-}$ be its Cuntz splice. Let $L_n=L(\cR_n)$ and $L_{n^-}=L(\cR_{n^-})$ be the Leavitt path algebras over a unital ring $\ell$. Let $C_m$ be the cyclic group on $2\le m\le \infty$ elements. Equip $L_n$ and $L_{n^-}$ with their natural $C_m$-gradings.  We show that under mild conditions on $\ell$, which are satisfied for example when $\ell$ is a field or a PID, there are no unital $C_m$-graded ring homomorphisms $L_n\to L_{n^-}$ nor in the opposite direction. 
\end{abstract}
\maketitle

\section{Introduction}
Let $\ell$ be a unital ring and $n\ge 1$. The \emph{Leavitt algebra} over $\ell$ is the Leavitt path algebra \cite{lpabook}*{Definition 1.2.3} $L_n=\ell\otimes L_{\Z}(\cR_n)$ of the graph $\cR_n$ consisting of a single vertex and $n$ loops. We write $L_{n-}$ for the Leavitt path algebra over $\ell$ of the graph $\cR_{n^-}$ whose adjacency matrix is

\begin{equation}\label{intro:Rn-}
A_{\cR_{n^-}} = \begin{pmatrix}
n & 1 & 0\\
1 & 1 & 1\\
0 & 1 & 1    
\end{pmatrix}.  
\end{equation}
The graph $\cR_{n^-}$ is the Cuntz splice of $\cR_n$ \cite{alps}*{Defintion 2.11}. 
It is an open question \cite{alps}*{Hypothesis on page 24} whether the algebras $L_2$ and $L_{2^-}$ over a field $\ell$ are isomorphic or not. As with any Leavitt path algebras, $L_n$ and $L_{n^-}$ are graded over the infinite cyclic group $C_\infty$, and therefore also over the cyclic
group $C_m$ of $m$ elements for all $m\ge 2$, via the grading mod $m$. The main result of this paper is the following theorem, which puts together Theorems \ref{thm:main1} and \ref{thm:main2}.

\begin{thm}\label{intro:main}
Let $n\ge 2$ and $2\le m\le \infty$. Assume that $\ell$ is regular supercoherent and that the canonical map $\Z\to K_0(\ell)$ is an isomorphism. Then there are no unital $C_m$-graded ring homomorphisms $L_n\to L_{n^-}$ nor in the opposite direction. 
\end{thm}
Theorem \ref{intro:main} generalizes a similar statement proved for $n=m=2$ in \cite{classinvo}*{Proposition 5.6}. It implies that $L_n$ and $L_{n^-}$ are not graded isomorphic over any of the cyclic groups; this was well-known for the infinite cyclic group (see e.g. \cite{towards}*{Example 4.2}, \cite{talent}*{end of Section 4.1}), and follows from \cite{classinvo}*{Proposition 5.6} in the case $n=m=2$. The hypothesis on $\ell$ in the theorem above are satisfied, for example, when $\ell$ is a field, or a PID, or a noetherian regular local ring. They guarantee that for any directed graph $E$ with finitely many vertices and edges and such that every vertex emits at least one edge, and any $2\le m\le\infty$, the Grothendieck group
$K_0^{C_m-\gr}(L(E))$ of $C_m$-graded, finitely generated projective modules, equipped with the shift action, is isomorphic to the Bowen-Franks $\Z[C_m]$-module 
\begin{equation}\label{intro:k0gr}
K_0^{C_m-\gr}(L(E))\cong\BF_m(E):=\coker(I-\tau_m A_E^t).
\end{equation}
Here $\tau_m$ is a generator of $C_m$ and $A_E$ is the adjacency matrix of $E$. Under the isomorphism above, the class of the free module of rank one with its standard grading is mapped to $[1]_E:=\sum_{v\in E^{0}}[v]\in\BF_m(E)$. Thus if $F$ is another such graph, the existence of a unital, $C_m$-graded ring homomorphism $L(E)\to L(F)$ implies the existence of a homomorphism of $\Z[C_m]$-modules $\BF_m(E)\to \BF_m(F)$ mapping $[1]_E\mapsto [1]_F$. Our proof of Theorem \ref{intro:main} consists in showing that, for $\tau=\tau_m$, there are no such maps between
\begin{equation*}\label{intro:bfrnrn-}
\BF_m(\cR_n)=\Z[C_m]/\langle 1-n\tau \rangle  \text{ and } \BF_m(\cR_{n^{-}})=\Z[C_m]/\langle \tau^3+(2n-1)\tau^2-(n+2)\tau+1\rangle.
\end{equation*}

The rest of this article is organized as follows. In Section \ref{sec:BF} the Bowen-Franks modules are introduced and some of their basic properties are proved, including a nontriviality criterion (Lemma \ref{lem:nonzerobfg}).  The latter is used in Section \ref{sec:bfrn} to establish a lower bound on the number of elements of $\BF_m(\cR_{n^-})$ in Proposition \ref{prop:bfg-sp-nontriv}. The modules $\BF_m(\cR_n)$ and $\BF_m(\cR_{n^-})$ are also computed in this section, in Lemmas  \ref{lem:bfrn} and \ref{lem:diag-splice}. The isomorphism \eqref{intro:k0gr} is proved in Section \ref{sec:k0gr} as Lemma \ref{lem:k0}. In Section 
\ref{sec:ln-toln} we prove Theorem \ref{thm:main1}, which says that if $n\ge 2$ and $2\le m\le\infty$ then there is no $C_m$-graded unital ring homomorphism $L_{n^-}\to L_n$. The nonexistence of $C_m$-graded unital homomorphisms in the opposite direction is established in Section \ref{sec:lntoln-} as Theorem \ref{thm:main2}.

\begin{ack} The computational results 
of Sections \ref{sec:ln-toln} and \ref{sec:lntoln-} were verified 
using SageMath \cite{sagemath}. This software
was also used to obtain Equation \eqref{eq:(n-1)^2} in Proposition \ref{prop:genej}. The first named author wishes to thank Lucas De Amorin
and Fernando Martin for useful discussions. The second named author wishes to thank Pere Ara, whose manifestation of interest in the particular case $n=m=2$ of our main theorem proved in \cite{classinvo} was a key encouragement to pursue the general case. 
\end{ack}

\begin{nota}\label{nota:natu} In this paper the natural numbers do not include $0$. We write $\N=\Z_{\ge 1}$ and $\N_0=\N\sqcup\{0\}$. For $m\in \N_{\ge 2}$, we write $C_m$ for the cyclic group of order 
$m$ and $C_\infty \simeq \Z$. Having fixed $m \in \N_{\geq 2} \cup \{\infty\}$, 
the symbol $\tau$ will refer to a generator of $C_m$, written multiplicatively.
\end{nota}

\section{The Bowen-Franks modules of a graph}\label{sec:BF}

A (directed) \emph{graph} $E$ consists of a set $E^0$ of vertices and a set $E^1$ of edges together with source and range functions $r,s:E^1\to E^0$. A vertex $v\in E^0$ is \emph{regular} if it emits a finite, positive number of edges; we write $\reg(E)\subset E^0$ for the set of regular vertices. The (reduced) \emph{adjacency matrix} of a graph $E$ is the matrix $A_E$ with nonnegative integer coefficients, indexed by $\reg(E)\times E^0$, whose $(v,w)$ entry is the number of edges with source $v$ and range $w$. The graph $E$ is \emph{regular} if $E^0=\reg(E)$ and \emph{finite} if both $E^0$ and $E^1$ are. If $E$ is both finite and regular, then $A_E$ is a finite square matrix with no zero rows. 

\begin{defn} Let $m\in \N_{\geq 2} \cup \{\infty\}$. The \emph{Bowen-Franks $\Z[C_m]$-module}
of a finite regular graph $E$ with adjacency matrix $A_E$ 
is $\BF_m(E) := \coker(I-\tau \cdot A_E^t)$. 
\end{defn}

The following lemma shall be useful in what follows.

\begin{lem}\label{lem:bfan}
Let  $2\le m<\infty$ and let $E$ be a finite regular graph. Equip $\coker(I-(A_E^t)^m)$ with the $C_m$-action $\tau \cdot [x] = [(A_E^t)^{m-1}x]$. Then there is a $\Z[C_m]$-module isomorphism
\[
\BF_m(E)\overset{\sim}{\longrightarrow}\coker(I-(A_E^t)^m),\,\,\, [v]\mapsto [v].
\]
\end{lem}
\begin{proof}
Put $M=\tau^{-1}(I-\tau A_E^t)$; clearly $\BF_m(E)\cong \coker(M)$ as $\Z[C_m]$-modules. The lemma is immediate from the form of the matrix of multiplication by $M$ with respect to the $\Z$-linear basis $\{v\tau^i:v\in E^0, 0\le i\le m-1\}$ of  $\Z[C_m]^{E^0}$. 
\end{proof}
\begin{coro}\label{coro:bfan}
The following are equivalent.
\item[i)] $\BF_m(E)$ is finite.
\item[ii)] $\chi_{A_E^m}(1)\ne 0$.
\goodbreak
\noindent If these equivalent conditions hold, then $|\BF_m(E)|=|\chi_{A_E^m}(1)|$.
\end{coro}
\begin{proof} Straightforward from Lemma \ref{lem:bfan} using the Smith normal form of the matrix $I-(A^t_E)^m$.
\end{proof}

Let $E$ be a finite regular graph. In the following lemma and elsewhere, we write
\begin{equation}\label{eq:chie}
\chi_E(x)=\det(xI-A_E)\in\Z[x]
\end{equation}
for the characteristic polynomial associated to the adjacency matrix of $E$.
 \begin{lemma}\label{lem:nonzerobfg} Let $E$ be a finite regular graph. Assume that all complex roots of $\chi_E$ are real. If $\BF_2(E)$ is finite and nontrivial, then
$\infty>|\BF_m(E)|> |\BF_2(E)|> 1$ for all $m>2$.
\end{lemma}
\begin{proof} Let $\chi_E = (x - \alpha_1) \cdots (x-\alpha_k)$ be the
irreducible factorization of $\chi_E$ in $\C[x]$ (repetitions are allowed). For each $m \geq 2$ we have
\begin{align*}
    |\chi_{A_E^m}(1)| = \prod_{i=1}^k |1-\alpha_i^m| &\geq 
    \prod_{i=1}^k |1-|\alpha_i|^m|. 
\end{align*} 
When $m = 2$ we have $|\alpha_i|^2 = \alpha_i^2$, since all roots are real, and thus 
the inequality above is in fact an equality. In particular, by Corollary \ref{coro:bfan}, the hypothesis that 
$\BF_2(E)$ is finite and nontrivial implies that $|\alpha_i|^m\ne 1$ for all $m\ge 1$. Hence the right hand side of the inequality is strictly increasing,
which shows that 
$|\BF_m(E)| = |\chi_{A_E^m}(1)| > |\BF_2(E)| > 1$ for $m > 2$ as desired.
\end{proof}

\section{The Bowen-Franks groups of the rose and of its Cuntz splice}\label{sec:bfrn}

Let $n \geq 1$; the \emph{rose of $n$ petals}  is the graph $\cR_n$ which consists of one
vertex and $n$ loops. The \emph{Cuntz splice} of $\cR_n$ (\cite{alps}*{Definition 2.11}) is the graph $\cR_{n^-}$ with adjacency matrix \eqref{intro:Rn-}.

\begin{lem}\label{lem:bfrn} Let $n,m\ge 1$. Then there is an isomorphism of $\Z[C_m]$-modules
\begin{equation}\label{eq:bfLn}
    \BF_m(\cR_n) \simeq \Z/(n^m-1)\Z.
\end{equation}
Here, the generator $\tau$ of $C_m$ acts on the right hand side of \eqref{eq:bfLn} as multiplication by $n^{m-1}$. 
\end{lem}
\begin{proof} Straightforward calculation. 
\end{proof}

Next we turn our attention to $\BF_m(\cR_{n^-})$. Before computing it, we may already use Lemma \ref{lem:nonzerobfg} to establish the following nontriviality result.

\begin{prop} \label{prop:bfg-sp-nontriv} Let $n,m \geq 2$. Then 
$\BF_m(\cR_{n^{-}})$ is finite and of order at least $3n^2-2n-1$.
\end{prop}
\begin{proof} A direct computation shows that 
\[
    \chi_{R_{n^-}}(x) = x^3 - (n+2)x^2 + (2n-1)x + 1.
\]
One checks that the signs of the values of this polynomial at  $-1, 1, 2$
and $n+1$ alternate; hence all of its roots are real. Likewise, computing
\[
    \chi_{A_{R_{n^-}}^2}(x) = x^3 - (6+n^2)x^2 + (4n^2-2n+5)x - 1
\]
we see that $|\chi_{A_{R_{n^-}}^2}(1)| = |3n^2-2n-1| = 3n^2-2n-1 > 1$ if $n \geq 2$. The result now follows
from Lemma \ref{lem:nonzerobfg}. 
\end{proof}

\begin{lem}\label{lem:diag-splice} 
Let $n\ge 1$ and $2\le m\le \infty$; put
\begin{equation}\label{eq:elxi}
\xi_n(x) := x^3 + (2n-1)x^2 - (n+2)x+1 \in \Z[x].
\end{equation}
There is an isomorphism of $\Z[C_m]$-modules 
\begin{equation}
    \label{eq:bfspn}
    \BF_m(L_{n^-}) \simeq \frac{\Z[C_m]}{\langle \xi_n(\tau) \rangle}
\end{equation}
that sends $[(1,1,1)]$ to $[1-n\tau]$. 
\end{lem}
\begin{proof} Let $B = A_{R_{n^-}}^t$; one checks that the matrices
\[
    R =\begin{pmatrix}
        1 & -n & 0\\
        \tau & 1-n\tau & (n-1)\tau - (n+1)\\
        \tau^2 & \tau-n\tau^2 & (n-1)\tau^2 - (n+1)\tau + 1 
        \end{pmatrix}    
\]
and
\[
C  =\begin{pmatrix}
     1 & (1-n)\tau + n & (n-1)\tau^3 + (2n^2-4n+1)\tau^2 + (1-3n^2)\tau + n(n+1)\\
     0 & 1 & (n+1) + (1-2n)\tau - \tau^2\\
     0 & 0 & 1
     \end{pmatrix}  
\]
are invertible in $M_3(\Z[C_m])$ and satisfy 
\[
R(I-\tau B)C =  \begin{pmatrix}
    1 & 0 & 0\\
    0 & 1 & 0\\
    0 & 0 & \tau^3+(2n-1)\tau^2-(n+2)\tau+1   
    \end{pmatrix}.
\] 
Hence we have a commutative diagram as follows, 
where all vertical arrows are isomorphisms. 
\begin{center}
\begin{tikzcd}[column sep = large]
    \Z[C_m]^3 \arrow{d}{C^{-1}}\arrow{r}{I-\tau B} &
     \Z[C_m]^3 \arrow{d}{R} \arrow{r} & \BF_n(L_{2^-}) \arrow[dashed]{d}{\overline{R}} \\
    \Z[C_m]^3 \arrow{r}[below]{R(I-\tau B)C}
    & \Z[C_m]^3 \arrow{r} & 
    \frac{\Z[C_m]^3}{\langle e_1, e_2, (\tau^3+(2n-1)\tau^2-(n+2)\tau+1)e_3\rangle}
    \arrow{d}{\pi_3}\\
    & & \frac{\Z[C_m]}{\langle \tau^3+(2n-1)\tau^2-(n+2)\tau+1 \rangle}
\end{tikzcd}
\end{center}
The desired isomorphism is the composition $\pi_3 \overline{R}$. Indeed,
from the fact that the quotient map $\Z[C_m]^3 \to \BF_m(L_{n^-})$
maps $(1,1,1)$ to $[(1,1,1)]$ and the commutativity of the diagram above, we conclude that 
\[
\pi_3\overline{R}([(1,1,1)]) = \pi_3([R(1,1,1)]) = \pi_3([(1-n,-n,1-n\tau)]) = [1-n\tau].    
\]
\end{proof}

\section{Bowen-Franks modules, graded \topdf{$K_0$}{K0}, and Leavitt path algebras}\label{sec:k0gr}

A unital ring $R$ is \emph{coherent} if the category of its finitely presented modules is abelian; it is \emph{supercoherent} if the polynomial ring $R[t_1,\dots,t_n]$ is coherent for every $n\ge 0$. We say that $R$ is \emph{regular} if every (right) $R$-module has finite projective dimension, and \emph{regular supercoherent} if it is both regular and supercoherent. This implies that $R[t_1,\dots,t_n]$ is regular for all $n\ge 1$, by the argument of \cite{abc}*{beginning of Section 7}. 

\begin{stan}
Throughout the rest of this paper, $\ell$ will be a unital, regular supercoherent ring such that the canonical map
\[
\Z\to K_0(\ell)
\]
is an isomorphism.
\end{stan}

We write $L(E)$ for the \emph{Leavitt path algebra} of $E$ over $\ell$ (\cite{lpabook}), which equals the tensor product $\ell\otimes L_{\Z}(E)$ with the Leavitt path algebra over $\Z$. The algebra $L(E)$ carries a canonical $C_\infty$-grading $L(E)=\bigoplus_{k\in\Z} L(E)_k$ that makes it a $C_\infty$-graded algebra. Hence $L(E)$ is also $C_m$-graded for every $m\ge 2$, where $L(E)_{\bar{i}}=\bigoplus_{k\in\Z} L(E)_{mk+i}$ $(\bar{i}\in \Z/m\Z\cong C_m)$.

In general, if  $G$ is a group and $R=\bigoplus_{g\in G}R_g$ a $G$-graded unital ring, we write $K_0^{G-\gr}(R)$ for the group completion of the monoid of finitely generated projective $G$-graded $R$-modules. Grading shifts equip $K_0^{G-\gr}(R)$ with a $\Z[G]$-module structure.

In the case when $\ell$ is a field, the following lemma is immediate from a very particular case of \cite{gradstein}*{Proposition 5.7}. 
\begin{lem}\label{lem:k0}
Let $2\le m\le\infty$ and let $E$ be a finite regular graph. Then there is an isomorphism of $\Z[C_m]$-modules
\[
K_0^{C_m-\gr}(L(E))\cong \BF_m(E)
\]
which maps $[L(E)]\mapsto \sum_{v\in E^0}[v]$.
\end{lem}
\begin{proof}
Let $E_m$ be the $C_m$-cover of $E$ as defined in \cite{gradstein}*{Section 5.2}. Observe that $C_m$ acts on $L(E_m)$ by algebra automorphisms, making
$K_0(L(E_m))$ into a $C_m$-module. By \cite{gradstein}*{Proposition 2.5 and Corollary 5.3}, there is a $\Z[C_m]$-module isomorphism 
$K_0^{C_m-\gr}(L(E))\cong K_0(L(E_m))$. By \cite{abc}*{Theorem 7.6}, $K_0(L(E_m))=\coker(I-A_{E_m}^t)$; one checks that this $\Z[C_m]$-module is precisely $\BF_m(E)$.
\end{proof}
\begin{coro}\label{coro:k0}
Let $E$ and $F$ be finite regular graphs and let $m\ge 2$. Then any unital, $C_m$-graded ring homomorphism $L(E)\to L(F)$ induces a homomorphism of $\Z[C_m]$-modules $\BF_m(E)\to \BF_m(F)$ that maps $\sum_{v\in E^0}[v]\mapsto \sum_{w\in F^0}[w]$.\qed
\end{coro}

In the rest of this article, we shall concentrate on the Leavitt path algebras $L_n=L(\cR_n)$ and $L_{n^-}=L(\cR_{n^-})$. 

\section{Nonexistence of graded homomorphisms \texorpdfstring{$L_{n^-}\to L_n$}{Ln- -> Ln}}\label{sec:ln-toln}

\begin{thm}\label{thm:main1}
Let $n\in\N_{\ge 2}$ and $m\in\N_{\ge 2}\cup\{\infty\}$. Then there is no unital $C_m$-graded
ring homomorphism $L_{n^{-}} \to L_n$. 
\end{thm} 
\begin{proof} Because any $C_\infty$-graded homomorphism is $C_2$-graded, we may assume that $m<\infty$. Suppose that there exists a $C_m$-graded unital morphism as in the theorem. Then by Lemmas \ref{lem:bfrn} and \ref{lem:diag-splice} and Corollary \ref{coro:k0}, there must be a homomorphism of abelian groups 
$$
\phi:\Z[C_m]/\langle \xi_n(\tau)\rangle\to \Z/\langle n^m-1\rangle
$$
mapping $[1-n\cdot \tau]\mapsto [1]$ and sending multiplication by $\tau$ into multiplication by $n^{m-1}$. But 
then 
\[[1]=\phi([1-n\cdot \tau])=\phi(1)\cdot (1-n\tau)=\phi(1)(1-n^m)=0,\]
a contradiction. 
\end{proof}

\section{Nonexistence of graded homomorphisms \texorpdfstring{$L_n \to L_{n^-}$}{Ln->Ln-}}\label{sec:lntoln-}

Let $n,m\ge 2$. Consider the ideal
\begin{equation}\label{def:Inm}
I_{m,n} := \langle x^m-1, \xi_n(x) \rangle \triangleleft \Z[x].
\end{equation}

\begin{lem}\label{lem:critzx}
Let $n, m \geq 2$. The existence of  
a unital $C_m$-graded ring homorphism $L_{n} \to L_{n^-}$ would imply that $(1-nx)^2\in I_{m,n}$.
\end{lem}
\begin{proof}
Lemma \ref{lem:diag-splice} says that $\BF_m(\cR_{n^-})=\Z[x]/I_{m,n}$. The present lemma follows from this using Lemma \ref{lem:bfrn} and Corollary \ref{coro:k0}.
\end{proof}

\begin{prop}\label{prop:genej}
Let $m,n\ge 2$.  Assume that $(1-nx)^2\in I_{m,n}$. Then
\begin{equation*}\label{eq:char-inm}
    I_{m,n} = \langle (x-1)^2, m(x-1), (3n-1)(x-1) + n-1\rangle.
\end{equation*}
\end{prop}
\begin{proof}
The proof involves several steps, as follows.

\goodbreak

\noindent{\emph{Step 1:} $(n-1)^2\in I_{m,n}$.} Let
\begin{align*}
    a_n &= -n^5 + 2n^4 - 2n^3 + 3n^2,\\
    b_n &= -2n^6+5n^5-7n^4+10n^3-4n^2+2n,\\
    c_n &= n^6-2n^5+3n^4-4n^3+n^2-2n+1.
\end{align*}
Set $p_n(x) = (n^7-2n^6+2n^5-3n^4)x - n^6+2n^5-3n^4+4n^3$ and $q_n(x)
= a_n x^2 + b_n x + c_n$. A calculation shows that
\begin{equation}\label{eq:(n-1)^2}
    p_n(x)\xi_n(x) + q_n(x)(1-n \cdot x)^2 = (n-1)^2.
\end{equation}    

\medskip

\noindent{\emph{Step 2:} $(n+1)x-2\in I_{m,n}$.} Put
\[
q_n(x) = nx+2n^2-n+2, \,\, r_n(x) = (-n^4+2n^3-2n^2+3n)x+n^3-2n^2+n-2.    
\]
One checks that 
\[
    n^3 \cdot \xi_n(x) = (1-n\cdot x)^2 q_n(x) + r_n(x).
\]
This step is concluded once we observe that $r_n \equiv (n+1)x-2\pmod{(n-1)^2}$.

\medskip

\noindent{\emph{Step 3:} $(x-1)^2 \in I_{m,n}$.}
 By the previous steps, 
$1-nx \equiv x-1$ and so $0 \equiv (1-nx)^2 \equiv (x-1)^2$.

\medskip

\noindent{\emph{Step 4:} Conclusion.} Observe that $p \equiv (x-1) \cdot p'(1) + p(1)\pmod{(x-1)^2}$ 
for any $p \in \Z[x]$. Applying this to $x^m-1$ and $\xi_n$, we get the proposition.
\end{proof}

\bigskip

\begin{thm}\label{thm:main2} Let 
$m\in\N_{\ge 2}\cup\{\infty\}$ and let $n\ge 2$. Then there is no $C_m$-graded unital ring homomorphism
$L_n \to L_{n^{-}}$.
\end{thm}
\begin{proof} 
As in the proof of Theorem \ref{thm:main1}, we may assume that $m<\infty$. Furthermore, we can reduce ourselves to 
the case in which $m$ is prime, since for each $d \mid m$
a $C_m$-graded ring homomorphism $L_n \to L_{n^-}$ 
is also $C_d$-graded. 
Now, suppose that there exists a unital $C_m$-graded map $L_n\to L_{n^-}$. By Lemma \ref{lem:critzx}, $(1-nx)^2\in I_{m,n}$. Hence the indentity of Proposition \ref{eq:char-inm} tells us in particular that 
\[0 \equiv ((3n-1)(x-1) + n-1)(x-1) \equiv (n-1)(x-1) \pmod{I_{m,n}}.\]

First consider the case in which $m$ does not divide
$n-1$, which 
by primality of $m$ means that $m$ and $n-1$ are coprime.  
Hence there exist integers $s,t \in \Z$ such that 
$(n-1)s + mt = 1$ and so $(x-1) = s(n-1)(x-1) + tm(x-1) \equiv 0\pmod{I_{m,n}}$. 
Therefore
\[
I_{m,n} = \langle (x-1)^2, m(x-1), (3n-1)(x-1)+n-1, x-1\rangle 
= \langle x-1, n-1\rangle, 
\]
which in turn shows that $\BF_m(L_{n^-})\simeq \Z/(n-1)\Z$ 
as abelian groups. However, this is a contradiction; the bound 
of Proposition
\ref{prop:bfg-sp-nontriv} tells us that 
\[
|\BF_m(L_{n^-})| \geq |\BF_2(L_{n^-})| = 3n^2-2n-1 > n-1.
\]

We are left to prove the case in which the prime $m$ divides $n-1$. 
Write $n = am +1$, so that
\begin{align*}
    I_{m,n} = I_{m,am+1} &= \langle (x-1)^2, m(x-1), (3am+2)(x-1)+am\rangle \\
    &= \langle (x-1)^2, m(x-1), 2(x-1)+am\rangle.
\end{align*}
Setting $m = 2$, we obtain
\begin{align*}
    I_{2,2a+1} &= \langle (x-1)^2, 2(x-1), 2(x-1)+2a\rangle \\
    &= \langle (x-1)^2, 2(x-1), 2a\rangle.
\end{align*}
In particular we have an epimorphism
\[
\frac{(\Z/{2a}\Z)[x]}{\langle (x-1)^2 \rangle} \twoheadrightarrow \BF_2(\cR_{(2a+1)^-})
\]
which tells us that $|\BF_2(\cR_{(2a+1)^-})| \leq 4a^2$. This 
contradicts Proposition \ref{prop:bfg-sp-nontriv}, since
\[
    |\BF_2(\cR_{(2a+1)^-})| = 3(2a+1)^2-2(2a+1)-1 > (2a+1)^2 > 4a^2.
\]

We can therefore assume that $m$ 
is an odd prime. Write $m = 2b+1$ for some $b \geq 1$. Then
\[
I_{m,am+1} \ni m(x-1) - b(2(x-1)+am) = (x-1) - bam = x - (1+bam).
\]
This shows that
\begin{align*}
I_{m,am+1} &= \langle (x-1)^2, m(x-1), 2(x-1) + am, x-(1+bam)\rangle\\
        &= \langle b^2 a^2 m^2, bam^2, 2abm+am, x-(1+abm)\rangle.   
\end{align*}
Noting that $2abm+am = am(2b+1) = am^2$, we get 
$I_{m,am+1} = \langle am^2, x-(1+abm)\rangle$. Therefore, we obtain 
an isomorphism of abelian groups
\[
\BF_{m}(L_{(am+1)^-}) \simeq \Z[x]/\langle x-(1+abm), am^2\rangle \simeq \Z/{am^2}\Z.
\]
Once again we draw a contradiction from the
bound of Proposition \ref{prop:bfg-sp-nontriv}, since
\[
|\BF_{m}(L_{(am+1)^-})| \geq |\BF_{2}(L_{(am+1)^-})| = 3(am+1)^2-2(am+1)-1 \geq (am+1)^2 > am^2.
\]
This completes the proof.
\end{proof}
\begin{rmk} The case $n=2$ of Theorem \ref{thm:main2} does not need the full force of Proposition \ref{prop:genej}; it can be derived after Step 3 of its proof. Indeed the existence of a graded unital ring a homomorphism $L_2\to L_{2^-}$ 
would imply $1 = (2-1)^2 \equiv 0 \pmod{I_{m,2}}$, contradicting Proposition \ref{prop:bfg-sp-nontriv}.
\end{rmk}

\end{document}